\documentclass[11pt,reqno]{amsart}
\usepackage{amsmath,amssymb,amsthm}
\usepackage{mathrsfs}
\usepackage{graphics}
\usepackage{graphicx,color, xcolor, soul} 
\usepackage{epsfig}

\theoremstyle{plain}
\newtheorem{thm}{Theorem}[section]
\newtheorem{lem}{\textbf{Lemma}}[section]
\newtheorem{cor}[thm]{Corollary}

\theoremstyle{definition}
\newtheorem{definition}{Definition}[section]

\numberwithin{equation}{section}
\begin{document}
	\begin{center}
		{\LARGE 
			{\bf{Some 
					$q$-fractional order difference sequence spaces}}
			}
		\vspace{.5cm}
		
	Taja Yaying$^{1}$, Pinakadhar Baliarsingh$^{2}$, Bipan Hazarika$^{3,\ast}$
		
		\vspace{.2cm}
				$^{1}$Department of Mathematics, Dera Natung Government College, Itanagar 791113, India
				
					\vspace{.2cm}
					E-mail: tajayaying20@gmail.com\\
					
					$^{2}$KIT,  Institute of Mathematics and Applications, Bhubaneswar-751029, Odhisa, India.\\
					
						\vspace{.2cm}
					E-mail: pb.math10@gmail.com\\

			$^{3}$Department of Mathematics, Gauhati University, Guwahati-781014, Assam, India.\\
				\vspace{.2cm}
		Email:bh\_rgu@yahoo.co.in; bh\_gu@gauhati.ac.in
	\end{center}
	\title{}
	\author{}
	\thanks{{\today}, $^\ast$The corresponding author}

	\begin{abstract} This paper intends to develop a $q$-difference operator $\nabla^{(\gamma)}_q$ of fractional order $\gamma$, and give several intriguing properties of this new difference operator. Our main focus remains on the construction of sequence spaces $\ell_p(\nabla^{(\gamma)})$ and $\ell_\infty (\nabla^{(\gamma)})$, at the same time comparing these spaces with those already exist in the literature. Apart from obtaining Schauder basis, we determine $\alpha$-, $\beta$-, and $\gamma$-duals of the newly defined spaces. A section is also devoted for characterizing matrix classes $(\ell_p(\nabla^{(\gamma)}),\mathfrak X),$ where $\mathfrak X$ is any of the spaces $\ell_\infty,$ $c,$ $c_0$ and $\ell_1$.
		AMS Subject Classification No:40C05, 46A45, 46B45, 47B37.
	\end{abstract}
	\maketitle
	\section{Introduction and Backgrounds}
	The $q$-calculus has significantly contributed to the development of various mathematical areas, including combinatorics, algebra, approximation theory, calculus, special functions, integro-differential equations and hypergeometric functions. Recently, its applications have extended to summability \cite{AktugluBekar,MursaleenTabassumFatima} and sequence spaces \cite{DemirizSahin,YayingHazarikaMursaleen_q-CesaroSpace}, gaining attention among the mathematical researchers. We revisit some basic concepts that are commonly employed in $q$-theory:
	
	Let $\mathbb{Z}^+$ be the collection of all positive integers and $\mathbb{Z}^+=\mathbb{Z}\cup \{0\}$. The $q$-integer $[i]_q$ $(q\in (0,1))$ is given by the following expression:
	\begin{eqnarray*}
		[i]_q=\left\{\begin{array}{ccc}
			\sum\limits_{j=0}^{i-1}q^j,& & i\in \mathbb{Z}^+, \\
			0, &  & i=0.
		\end{array}\right.
	\end{eqnarray*}
	Evidently, $[i]_q=i$ as $q\to 1^-$.
	\begin{definition}
		Forany two $\mu,\nu\in \mathbb{Z}^+_0,$ the notation $\binom{\mu}{\nu}_{q}$ is given by
		\begin{eqnarray*}
			\binom{\mu}{\nu}_{q}=\left\{\begin{array}{ccc}
				\dfrac{[\mu]_q!}{[\mu-\nu]_q! [\nu]_q!},&  &\mu\geq \nu, \\
				0,& & \mu<\nu,\end{array}\right.
		\end{eqnarray*}
		and is referred to as the $q$-binomial coefficient, the $q$-analog of the binomial coefficient $\binom{\mu}{\nu}$. Here, $[\mu]_q! = \prod_{\nu=1}^{\mu}[\nu]_q $ means the $q$-analog of $\mu!$.
	\end{definition}
	
	Additionally, we have $\binom{0}{0}_q=\binom{\mu}{0}_q=\binom{\mu}{\mu}_q=1$ and $\binom{\mu}{\mu-\nu}_q=\binom{\mu}{\nu}_q$. For further study and various concepts related to $q$-theory, readers may refer to the monograph \cite{KacCheung} and the expository article \cite{Srivastava}. 
	\subsection{Sequence spaces}
	In the entirety of this article, the symbols $\mathbb{Z}^+$ and $\mathbb{C}$ are used to represent the collection of all positive integers and the set of all imaginary numbers, respectively. Let $\mathbb{Z}^+_0=\mathbb{Z}^+\cup \{0\}$. Define the set $\omega$ by
	$$\omega=\left\{g=(g_j): g_j \in \mathbb{C} \text{ for all $j\in \mathbb{Z}^+_0$}\right\}.$$
	The set $\omega$ when associated with the operations of addition and scalar multiplication of sequences, defined by
	$$(g_j)+(h_j)=(g+h)_j \mbox{~and~} \xi (g_j)=(\xi g_j)$$
	for all $g=(g_j),h=(h_j)\in \omega$ and $\xi\in \mathbb{C},$ forms a linear space. Any subspaces of $\omega$ preserving the above operations of sequences are known as sequence spaces. Some well known examples are the set $\ell_p,$ $c$, $c_0$, and $\ell_\infty$ defined by
	\begin{eqnarray*}
		\ell_p&:=&\left\{g=(g_j)\in \omega: \sum_{j=0}^\infty \vert g_j\vert^p <\infty \right\}, ~(0<p<\infty),\\
		c&:=&\left\{g=(g_j)\in \omega: \lim_{j\to \infty}  g_j=\alpha \text{ for some $\alpha\in \mathbb{C}$} \right\},\\
		c_0&:=&\left\{g=(g_j)\in \omega: \lim_{j\to \infty}  g_j=0 \right\},\\
		\ell_\infty&:=&\left\{g=(g_j)\in \omega: \sup_{j\in \mathbb{Z}^+_0} \vert g_j \vert<\infty \right\},
	\end{eqnarray*}
	and are often regarded as classical sequence spaces. A complete space that is equipped with a norm is called a Banach space, and a Banach space with continuous coordinate functionals forms a $BK$-space. The space $\ell_p,$ $(1<p<\infty),$ is a $BK$-space attached by the norm
	$$\left\|g \right\|_{\ell_p}=\left(\sum_{j=0}^\infty\vert g_j \vert^p\right)^{1/p},$$
	and for $0<p<1,$  $\ell_p$ forms a complete $p$-normed space attached by $p$-norm 
	$$\left\|g \right\|_{\ell_p}=\sum_{j=0}^\infty\vert g_j \vert^p.$$
	
	Let $\Phi=(\phi_{jk})$ be an infinite matrix of  imaginary entries $\phi_{jk}$ for all $j,k\in \mathbb{Z}^+_0$. When $\mathfrak X$ is a sequence space, the set $\mathfrak X_{\Phi}$ defined by
	$$\mathfrak X_{\Phi}:=\left\{g=(g_j)\in \omega: \Phi g\in \mathfrak X \right\}$$
	is again a sequence space, and is often called the domain of $\Phi$ in $\mathfrak X.$ Here, the notation $$\Phi g=\left\{ (\Phi g)_j \right\}_{j\in \mathbb{Z}^+_0}=\left\{\sum_{k=0}^\infty \phi_{jk}g_k \right\}_{j\in \mathbb{Z}^+_0}$$ is called $\Phi$-transform of the sequence $g=(g_j)$, where the notation $\Phi_j$ means the $j^{\textrm{th}}$ row of the matrix $\Phi.$ We further define the notation $(\mathfrak X, \mathfrak Y)$ for any two $\mathfrak X, \mathfrak Y\subset \omega$ in the following manner:
	$$(\mathfrak X, \mathfrak Y)=\left\{ \Phi:\mathfrak X \to \mathfrak Y: \Phi g\in \mathfrak Y \text{ for all } g\in \mathfrak X \right\}.$$
	
	We note that $\mathfrak X_{\Phi}$ is $BK$-space when $\mathfrak X$ itself is $BK$-space and $\Phi$ represents a triangular matrix, with the norm defined as $\left\| g \right\|_{\mathfrak X_{\Phi}} = \left\| \Phi g \right\|_{\mathfrak X}$. For extensive studies on the application of triangular matrices in classical sequence spaces, the popular books \cite{Basar_Monograph,MursaleenBasar_Monograph}, along with the references cited therein, are valuable resources.
	
	\subsection{Difference sequence spaces}
	We define difference operators $\Delta$ and $\nabla$ as follows:
	\[
	(\Delta g)_j = g_j - g_{j+1} \quad \text{and} \quad (\nabla g)_j = g_j - g_{j-1},
	\]
	for all $j \in \mathbb{Z}^+_0$, where $g_j = 0$ for $j < 0$. The operator $\Delta$ is known as the forward difference operator, and $\nabla$ is referred to as the backward difference operator, both of first order. The spaces $c_0(\Delta)$, $c(\Delta)$, and $\ell_{\infty}(\Delta)$ were defined by K\i zmaz \cite{Kizmaz} and later generalized by Et \cite{Et_DifferenceSeqnSpace} by developing matrix domains $\ell_{\infty}(\Delta^2)$, $c(\Delta^2)$, and $c_0(\Delta^2)$, where $\Delta^2$ is defined by $(\Delta^2 g)_j = (\Delta g)_j - (\Delta g)_{j+1}$ for all $j \in \mathbb{Z}^+_0$. The space $bv_p = (\ell_p)_{\nabla}$ for $1 < p < \infty$ was introduced by Ba\c sar and Altay \cite{BA}, and a similar space for $0 < p < 1$ was explored by Altay and Ba\c sar \cite{AB}.
	
	These difference operators are further extended to $\Delta^{r}$ and $\nabla^{(r)}$, $(r\in \mathbb{Z}^+_0),$ in the following manner:
	\[
	(\Delta^{(r)}g)_j=\sum_{k=0}^r (-1)^k \binom{r}{k}g_{j+k} \mbox{ and } (\nabla^{(r)}g)_j=\sum_{k=0}^r (-1)^k \binom{r}{k}g_{j-k},
	\]
	for all $j\in \mathbb{Z}^+_0.$ These are called forward and backward difference operators of $r^{\textrm{th}}$ order.   The matrix domains $\ell_{\infty}(\Delta^{(r)})$, $c_0(\Delta^{(r)})$, and $c(\Delta^{(r)})$ were investigated by Et and \c Colak \cite{EtColak_GeneralizedDiffSeqnSpaces}, while  $\ell_{\infty}(\nabla^{(r)})$, $c_0(\nabla^{(r)})$, and $c(\nabla^{(r)})$ were explored by Malkowsky and Parashar \cite{MalkowskyParashar}.
	
	Let $\gamma$ be a positive proper fraction. Baliarsingh and Dutta   \cite{BD} introduced an extensive fractional difference operator $\nabla^{(\gamma)}$ in the following manner:
	\begin{equation}\label{eqn3}
	(\nabla^{(\gamma)}g)_j=\sum_{k=0}^\infty (-1)^k\frac{\Gamma(\gamma+1)}{k!\Gamma(\gamma-k+1)}g_{j-k},
	\end{equation}
	given that the infinite sum is convergent. Alternatively, the operator $\nabla^{(\gamma)}$ may also be represented in form of a lower triangular matrix as follows:
	\begin{equation*}
	(\nabla^{(\gamma)})_{jk}=
	\begin{cases}
	(-1)^{j-k}\frac{\Gamma(\gamma+1)}{(j-k)!\Gamma(\gamma-j+k+1)}\quad & (0\leq k\leq j),\\
	0 & (k>j).
	\end{cases}
	\end{equation*}
	Dutta and Baliarsingh \cite{DB} explored paranormed fractional difference spaces $\mathfrak X(\Gamma,\Delta^{(\alpha)},u,p)$ for $\mathfrak X=\{\ell_\infty,c,c_0\},$ where the operator is defined as $$(\Delta^{(\gamma)}g)_j= \sum_{k=0}^\infty (-1)^k\frac{\Gamma(\gamma+1)}{k!\Gamma(\gamma-i+1)}g_{j+k}$$ for all $j\in \mathbb{Z}^+_0.$ Additionally, Baliarsingh and Dutta \cite{BD1} investigated the spaces $\mathfrak X(\Gamma, \Delta^{(\gamma)}, p)$ for $\mathfrak X = \{\ell_{\infty}, c, c_0\}$. \"Ozger \cite{Ozger} defined the sequence space $\ell_p(\nabla^{(\gamma)})$ and examined compactness in this space using Hausdorff measure of non-compactness. Yaying \cite{Yaying} analyzed paranormed Riesz difference spaces $\mathfrak X(p, \Delta^{(\gamma)})$ of fractional order $\gamma$, where $\mathfrak X = \{r^t_{\infty}, r^t_0, r^t_c\}$ denotes Riesz sequence spaces. 
	
	\subsection{Motivation}
	Recently some papers cropped up in the area of sequence spaces by utilizing $q$-theory. As an immediate example one may consult the papers due to Demiriz and \c Sahin \cite{DemirizSahin} and Yaying et al. \cite{YayingHazarikaMursaleen_q-CesaroSpace}, where in the authors introduced new spaces that are generated by $q$-analogue of Ces\`aro matrix. It is known that the sequence $(k)_{k=1}^\infty$ diverges to $+\infty$. However, the $q$-integers do not follow the same pattern. For instance, the $q$-sequence $([k]_q)_{k=1}^\infty$ converges to $\frac{1}{1-q}$ as $k\to \infty.$ This idea led to the development of new spaces by using $q$-theory.
	
	The $q$-calculus has made its place in the area of difference spaces too. For instance, Yaying et al. \cite{AlotaibiYayingMohiuddine, YayingHazarikaTripathyMursaleen} investigated the matrix domain $\mathfrak X_{\nabla^{(2)}_q},$ $\mathfrak X \in \{\ell_p, c_0, c, \ell_{\infty}\}$, where $\nabla^{(2)}_q$ denotes 2nd-order $q$-difference operator given by
	$$(\nabla^{(2)}_q g)_j = g_j - (1+q)g_{j-1} + qg_{j-2}.$$
	They also studied the spectral partition of the operator $\nabla^{(2)}_q$ over the spaces $c_0$, $c$, and $\ell_1$. The authors in \cite{AlotaibiYayingMohiuddine} also remarked that the $q$-difference operator $\nabla^{(1)}_q$ coincides with the classical backward difference operator $\nabla.$  
	
	A more generalized $q$-difference spaces $c_0(\nabla^{(r)}_q)$ and $c(\nabla^{(r)}_q)$ are introduced by Yaying et al. \cite{YHME}, where $\nabla^{(r)}_q$ is $r^{\textrm{th}}$ order $q$-difference operator defined as follows;
	$$
	(\nabla^{(r)}_q g)_j=\sum_{k=0}^r (-1)^kq^{\binom{k}{2}}\binom{j}{k}_qg_{j-k}, 
	$$
	for all $j\in \mathbb{Z}^+_0.$ The generalized nature of the operator $\nabla^{(r)}_q$ can be deduced from the following special cases:
	\begin{enumerate}
		\item The operator $\nabla^{(r)}_q$ is reduced to $\nabla^{(r)}$ when $q\to 1^-$, which further reduces to $\nabla^{(2)}$ and $\nabla$ when $r=2$ and $r=1$, respectively.
		\item  The operator $\nabla^{(r)}_q$ is reduced to $\nabla^{(2)}_q$ when $r=2.$
	\end{enumerate}
	Quite recently, Ellidokuzo\u glu and Demiriz \cite{ED} introduced the space $\ell_p(\Delta^{r}_q)$ and focussed on the duals and characterization of matrix classes pertaining to this space.
	
	Motivated by the studies discussed above, we wish to develop a more generalized difference operator $\nabla^{(\gamma)}_q$ and investigate its domain in the spaces $\ell_p$ and $\ell_\infty$. Our objectives include determination of spectrum of this new operator $\nabla^{(\gamma)}_q$ over the space $\ell_1$.
	
	\section{Fractional order $q$-difference operator $\nabla^{(\gamma)}_q$ and its domains}
	
	Before proceeding to our primary findings, we note down certain basic properties of $q$-gamma function due to \cite{Askey}:
	
	Let $q\in (0,1)$ and $t>0.$ Then, the $q$-gamma function of $t$ is given by
	$$\Gamma_q(t)=\dfrac{(q,q)_\infty}{(q^t,q)_\infty}(1-q)^{1-t},$$
	where the notation $(x,q)_\infty$ represents the product
	$$(x,q)_\infty=\prod_{j=0}^\infty (1-xq^j).$$
	The $q$-gamma function fulfils the following properties:
	\begin{enumerate}
		\item For a positive integer $t$, $\Gamma_q(t+1)=1(1+q)(1+q+q^2)\ldots (1+q+q^2+\ldots+q^{t-1})=[t]_q!.$ That is $\Gamma_q(t+1)$ is $q$-analogue of factorial of $t.$
		\item For $t>0,$ $\Gamma_q(t+1)=\dfrac{1-q^t}{1-q}\Gamma_q(t)=[t]\Gamma_q(t).$ Additionally, $\Gamma_q(1)=1.$
		\item For $t>0,$ $\Gamma_q(t)=\Gamma(t)$ as $q$ approaches $1^-.$
	\end{enumerate}
	
	Now, we define $q$-difference operator $\nabla^{(\gamma)}:\omega \to \omega$ of fractional order $\gamma$ in the following manner:
	\[
	(\nabla^{(\gamma)}_q g)_j=\sum_{k=0}^\infty (-1)^k q^{\binom{k}{2}}\dfrac{\Gamma_q(\gamma+1)}{[k]_q!\Gamma_q(\gamma-k+1)}g_{j-k}=g_j-[\gamma]_qg_{j-1}+q\frac{[\gamma]_q[\gamma-1]_q}{[2]_q!}g_{j-2}+\ldots, 
	\]
	for all $j\in \mathbb{Z}^+_0.$ It is evident that the operator $\nabla^{(\gamma)}_q$ results into several difference operators depending on the choice of $\gamma$. Few of them are:
	\begin{enumerate}
		\item For $\gamma=1/2,$ \begin{eqnarray*}(\nabla_q^{(1/2)}g)_j&=&\sum_{k=0}^\infty (-1)^k q^{\binom{k}{2}}\dfrac{\Gamma_q(1/2+1)}{[k]_q!\Gamma_q(1/2-k+1)}g_{j-k}\\
			&=&g_j-\left[\frac{1}{2}\right]_qg_{j-1}+q\frac{\Gamma_q(3/2)}{[2]_q!\Gamma_q(-1/2)}g_{j-2}+q^3\frac{\Gamma_q(3/2)}{[2]_q!\Gamma_q(-3/2)}g_{j-3}+\ldots
		\end{eqnarray*}
		\item For $\gamma=1,$ $$(\nabla_q^{(1)}g)_j=\sum_{k=0}^\infty (-1)^k q^{\binom{k}{2}}\dfrac{\Gamma_q(2)}{[k]_q!\Gamma_q(2-k)}g_{j-k}=g_j-g_{j-1}=(\nabla g)_j.$$
		\item For $\gamma=2,$ $$(\nabla_q^{(2)}g)_j=\sum_{k=0}^\infty (-1)^k q^{\binom{k}{2}}\dfrac{\Gamma_q(3)}{[k]_q!\Gamma_q(3-k)}g_{j-k}=g_j-[2]_q g_{j-1}+q g_{j-2}.$$
		\item For $\gamma=r\in \mathbb{Z}^+$,
		\begin{eqnarray*}
			(\nabla_q^{(\alpha)}g)_j&=&\sum_{k=0}^\infty (-1)^k q^{\binom{k}{2}}\dfrac{\Gamma_q(r+1)}{[k]_q!\Gamma_q(r-k+1)}g_{j-k}\\
			&=&\sum_{k=0}^\infty (-1)^k q^{\binom{k}{2}}\dfrac{[r]_q!}{[k]_q![r-k]_q!}g_{j-k}\\
			&=&\sum_{k=0}^r (-1)^k q^{\binom{k}{2}}\binom{r}{r-k}_qg_{j-k}.
		\end{eqnarray*}
	\end{enumerate}
	
	\begin{thm}
		The operator $\nabla^{(\gamma)}_q:\omega \to \omega$ is linear.
	\end{thm}
	\begin{proof}
		It cab be readily confirmed the facts that
		$$\left(\nabla^{(\gamma)}_q\left(g+h\right)\right)_j=\left(\nabla^{(\gamma)}_q g\right)_j+\left(\nabla^{(\gamma)}_q h\right) \text{ and } \left(\nabla^{(\gamma)}_q\left(cg\right)\right)_j=c\left(\nabla^{(\gamma)}_q g\right)_j$$
		for all $j\in \mathbb{Z}^+_0$, $g=(g_j),h=(h_j)\in \omega$ and $c\in \mathbb{C}$. Hence, we skip the detailed proof.
	\end{proof}
	
	\begin{thm}
		In general, $\nabla^{(\mu)}_q\cdot \nabla^{(\nu)}_q\neq \nabla^{(\mu+\nu)}_q,$ for $\mu,\nu\in \mathbb{R}.$
	\end{thm}
	\begin{proof} We cite an example to claim this statement:\\
		Choose $\mu=\nu=1/2.$ Then, keeping in mind the linearity of the operator $\nabla^{(1/2)}_q,$ we obtain that
		\begin{eqnarray*}
			&&\left(\left(\nabla^{(1/2)}_q\cdot \nabla^{(1/2)}_q\right) g\right)_j\\
			&=&\left(\nabla^{(1/2)}_q \left(\nabla^{(1/2)}_qg\right) \right)_j\\
			&=&\left(\nabla^{(1/2)}_qg\right)_j-\left[\frac{1}{2}\right]_q \left(\nabla^{(1/2)}_qg\right)_{j-1}+q\frac{\left[\frac{1}{2}\right]_q\left[\frac{-1}{2}\right]_q}{[2]_q!} \left(\nabla^{(1/2)}_qg\right)_{j-2}-q^3\frac{\left[\frac{1}{2}\right]_q\left[\frac{-1}{2}\right]_q\left[\frac{-3}{2}\right]_q}{[3]_q!} \left(\nabla^{(1/2)}_qg\right)_{j-3}+\\
			&\phantom{\mathrel{=}}& q^6\frac{\left[\frac{1}{2}\right]_q\left[\frac{-1}{2}\right]_q\left[\frac{-3}{2}\right]_q\left[\frac{-5}{2}\right]_q}{[4]_q!} \left(\nabla^{(1/2)}_qg\right)_{j-4}-\ldots\\
			&=&\left(g_j-\left[\frac{1}{2}\right]_qg_{j-1}+q\frac{\left[\frac{1}{2}\right]_q\left[\frac{-1}{2}\right]_q}{[2]_q!}g_{j-2}-q^3\frac{\left[\frac{1}{2}\right]_q\left[\frac{-1}{2}\right]_q\left[\frac{-3}{2}\right]_q}{[3]_q!}g_{j-3}+\ldots\right)-\left[\frac{1}{2}\right]_q\\
			&\phantom{\mathrel{=}}& \left(g_{j-1}-\left[\frac{1}{2}\right]_qg_{j-2}+q\frac{\left[\frac{1}{2}\right]_q\left[\frac{-1}{2}\right]_q}{[2]_q!}g_{j-3}-q^3\frac{\left[\frac{1}{2}\right]_q\left[\frac{-1}{2}\right]_q\left[\frac{-3}{2}\right]_q}{[3]_q!}g_{j-4}+\ldots\right)+q\frac{\left[\frac{1}{2}\right]_q\left[\frac{-1}{2}\right]_q}{[2]_q!}\\
			&\phantom{\mathrel{=}}&\left(g_{j-2}-\left[\frac{1}{2}\right]_qg_{j-3}+q\frac{\left[\frac{1}{2}\right]_q\left[\frac{-1}{2}\right]_q}{[2]_q!}g_{j-4}-q^3\frac{\left[\frac{1}{2}\right]_q\left[\frac{-1}{2}\right]_q\left[\frac{-3}{2}\right]_q}{[3]_q!}g_{j-5}+\ldots \right)-q^3\frac{\left[\frac{1}{2}\right]_q\left[\frac{-1}{2}\right]_q\left[\frac{-3}{2}\right]_q}{[3]_q!}\\
			&\phantom{\mathrel{=}}&\left( g_{j-3}-\left[\frac{1}{2}\right]_qg_{j-4}+q\frac{\left[\frac{1}{2}\right]_q\left[\frac{-1}{2}\right]_q}{[2]_q!}g_{j-5}-q^3\frac{\left[\frac{1}{2}\right]_q\left[\frac{-1}{2}\right]_q\left[\frac{-3}{2}\right]_q}{[3]_q!}g_{j-6}+\ldots\right)+\ldots\\
			&=&g_j-2\left[\frac{1}{2}\right]_q g_{j-1}+\frac{1+5q-4q^{1/2}}{(1+q)(1-q)^2}g_{j-2}+\ldots\\
			&\neq& \left(\nabla^{(1)}_qg\right)_,~j \forall j\in \mathbb{Z}^+_0. 
		\end{eqnarray*}
			\end{proof}
	Now, let us introduce the operator $\nabla^{(-\gamma)}_q:\omega\to \omega$ defined in the following manner:
	\[
	(\nabla^{(-\gamma)}_q g)_j=\sum_{k=0}^\infty (-1)^k \dfrac{\Gamma_q(1-\gamma)}{[k]_q!\Gamma_q(1-\gamma-k)}g_{j-k}=g_j+[\gamma]_qg_{j-1}+\frac{[\gamma]_q[\gamma+1]_q}{[2]_q!}g_{j-2}+\ldots, 
	\]
	for all $j\in \mathbb{Z}^+_0.$
		\begin{lem}\label{InverseLemma}
		The operator $\nabla^{(-\gamma)}_q$ is the inverse of the operator $\nabla^{(\gamma)}_q.$ That is to say that
		$$\nabla^{(\gamma)}_q \nabla^{(-\gamma)}_q =I=\nabla^{(-\gamma)}_q \nabla^{(\gamma)}_q,$$
		where $I$ is the identity operator. 
	\end{lem}
	\begin{proof}
		We only show the one sided equality. The other one can be obtained in a similar pattern.\\
		\begin{eqnarray*}
			&&\left(\left(\nabla^{(\gamma)}_q \nabla^{(-\gamma)}_q\right)g\right)_j\\
			&=&\left(\nabla^{(\gamma)}_q \left(\nabla^{(-\gamma)}_qg\right)\right)_j\\
			&=&\left(\nabla^{(\gamma)}_q \left( g_j+[\gamma]_q g_{j-1}+\frac{[\gamma]_q[\gamma+1]_q}{[2]_q!} g_{j-2}+\frac{[\gamma]_q[\gamma+1]_q[\gamma+2]_q}{[3]_q!} g_{j-3}+\ldots\right) \right)_j\\
			&=&\left(\nabla^{(\gamma)}_qg\right)_j+[\gamma]_q\left(\nabla^{(\gamma)}_qg\right)_{j-1}+\frac{[\gamma]_q[\gamma+1]_q}{[2]_q!}\left(\nabla^{(\gamma)}_qg\right)_{j-2}+\frac{[\gamma]_q[\gamma+1]_q[\gamma+2]_q}{[3]_q!}\left(\nabla^{(\gamma)}_qg\right)_{j-3}+\ldots\\
			&=&\left(g_j-[\gamma]_qg_{j-1}+q\frac{[\gamma]_q[\gamma-1]_q}{[2]_q!}g_{j-2}-\ldots \right)+[\gamma]_q\left(g_{j-1}-[\gamma]_qg_{j-2}+q\frac{[\gamma]_q[\gamma-1]_q}{[2]_q!}g_{j-3}-\ldots \right)+\\
			&\phantom{\mathrel{=}}&\frac{[\gamma]_q[\gamma+1]_q}{[2]_q!}\left(g_{j-2}-[\gamma]_qg_{j-3}+q\frac{[\gamma]_q[\gamma-1]_q}{[2]_q!}g_{j-4}-\ldots \right)+\frac{[\gamma]_q[\gamma+1]_q[\gamma+3]_q}{[3]_q!}\\
			&\phantom{\mathrel{=}}&\left(g_{j-3}-[\gamma]_qg_{j-4}+q\frac{[\gamma]_q[\gamma-1]_q}{[2]_q!}g_{j-5}-\ldots \right)+\ldots\\
			&=& g_j+\left(-[\gamma]_q+[\gamma]_q\right)g_{j-1}+\left( q\frac{[\gamma]_q[\gamma-1]_q}{[2]_q!}-[\gamma]^2_q+\frac{[\gamma]_q[\gamma+1]_q}{[2]_q!}\right)g_{j-2}+\\
			&\phantom{\mathrel{=}}&\left(-q^3\frac{[\gamma]_q[\gamma-1]_q[\gamma-2]_q}{[3]_q!}+q\frac{[\gamma]^2_q[\gamma-1]}{[2]_q!}-\frac{[\gamma]^2[\gamma+1]_q}{[2]_q!}+\frac{[\gamma]_q[\gamma+1]_q[\gamma+2]_q}{[3]_q!} \right)g_{j-3}+\ldots
		\end{eqnarray*}
		It is clear that the coefficient of $g_{j-1}$ is $0.$ The coefficient of $g_{j-2}$
		\begin{eqnarray*}
			&=&q\frac{[\gamma]_q[\gamma-1]_q}{[2]_q!}-[\gamma]^2_q+\frac{[\gamma]_q[\gamma+1]_q}{[2]_q!}\\
			&=&\frac{[\gamma]_q}{[2]_q!}\left(q[\gamma-1]_q-[2]_q![\gamma]_q+[\gamma+1]_q \right)\\
			&=&\frac{[\gamma]_q}{[2]_q!}\left(-[\gamma+1]_q+[\gamma+1]_q \right)\\
			&=&0.
		\end{eqnarray*} 
		It turns out that all the coefficients of $g_k$ is zero for $k=0$ to $j-1.$ Thus,
		$$\left(\left(\nabla^{(\gamma)}_q \nabla^{(-\gamma)}_q\right)g\right)_j=g_j.$$
		This completes the proof.
	\end{proof}
	
	It is apparent that the $q$-difference operator $\nabla^{(\gamma)}_q=(d^{\gamma}_{jk})$ can be represented in the form of a lower triangular matrix in the following manner:
	\begin{eqnarray*}
		d^{\gamma}_{jk}(q)=\left\{\begin{array}{ccc}
			(-1)^{j-k}q^{\binom{j-k}{2}}\dfrac{\Gamma_q(\gamma+1)}{[j-k]_q!\Gamma_q(\gamma-j+k+1)}&, & 0\leq k\leq j, \\
			0& , & k>j.
		\end{array}\right.
	\end{eqnarray*}
	Thus, for $g=(g_j)\in \omega,$ its $\nabla^{(\gamma)}_q$-transform is defined by $h=(h_j)\in \omega$ in the following manner:
	\begin{equation}\label{NablaTransform}
	h_j=\sum_{k=0}^{j}(-1)^{j-k}q^{\binom{j-k}{2}}\dfrac{\Gamma_q(\gamma+1)}{[j-k]_q!\Gamma_q(\gamma-j+k+1)}g_k
	\end{equation}
	for each $j\in \mathbb{Z}^+_0.$ Alternatively, by using Lemma \ref{InverseLemma}, the inverse $\nabla^{(-\gamma)}_q$-transform of the sequence $h=(h_j)\in \omega$ is defined in terms of $g=(g_j)\in \omega$ by
	\begin{equation}\label{InverseNablaTransform}
	g_j=\sum_{k=0}^{j}(-1)^{j-k}\dfrac{\Gamma_q(-\gamma+1)}{[j-k]_q!\Gamma_q(-\gamma-j+k+1)}h_k,
	\end{equation}
	for all $j\in \mathbb{Z}^+_0.$
	
	Taking in account of these facts, we shift our attention to introduce the sequence spaces $\ell_p(\nabla^{(\gamma)}_q)$ and $\ell_\infty(\nabla^{(\gamma)}_q)$ as follows:
	{\small
		\begin{eqnarray*}
			\ell_p(\nabla^{(\gamma)}_q)&=&\left\{g=(g_j)\in \omega: h=\nabla^{(\gamma)}_q g=\left(\sum_{k=0}^{j}(-1)^{j-k}q^{\binom{j-k}{2}}\dfrac{\Gamma_q(\gamma+1)}{[j-k]_q!\Gamma_q(\gamma-j+k+1)}g_k \right)_{j\in \mathbb{Z}^+_0}\in \ell_p \right\}\\
			\ell_\infty(\nabla^{(\gamma)}_q)&=&\left\{g=(g_j)\in \omega: h=\nabla^{(\gamma)}_q g=\left(\sum_{k=0}^{j}(-1)^{j-k}q^{\binom{j-k}{2}}\dfrac{\Gamma_q(\gamma+1)}{[j-k]_q!\Gamma_q(\gamma-j+k+1)}g_k \right)_{j\in \mathbb{Z}^+_0}\in \ell_\infty \right\}.
		\end{eqnarray*}
	}
	The above sequence spaces reduce into some special sequence spaces already exist in the literature depending on the choice of $\gamma$ and $q$:
	\begin{enumerate}
		\item For $\gamma=1,$ $\ell_p(\nabla^{(\gamma)}_q)=bv_p$ and $\ell_\infty(\nabla^{(\gamma)}_q)=bv_\infty$, where $bv_p$ and $bv_\infty$, $1\leq p\leq \infty,$ as studied by Ba\c sar and Altay \cite{BA}.
		\item For $\gamma=2,$ $\ell_p(\nabla^{(\gamma)}_q)=\ell_p(\nabla^{(2)}_q)$ and $\ell_\infty(\nabla^{(\gamma)}_q)=\ell_\infty(\nabla^{(2)}_q)$,  as studied by Alotaibi et al. \cite{AlotaibiYayingMohiuddine}.
		\item For $\gamma=r\in \mathbb{Z}^+_0$, $\ell_p(\nabla^{(\gamma)}_q)=\ell_p(\nabla^{(r)}_q)$ and $\ell_\infty(\nabla^{(\gamma)}_q)=\ell_\infty(\nabla^{(r)}_q)$, as studied by Ellidokuzo\u glu and Demiriz \cite{ED}.
		\item For $\gamma=r\in \mathbb{Z}^+_0$ and $q=1,$ $\ell_p(\nabla^{(\gamma)}_q)=\ell_p(\nabla^{(r)})$ and $\ell_\infty(\nabla^{(\gamma)}_q)=\ell_\infty(\nabla^{(r)})$, as studied by Altay \cite{Altay}.
		\item For $q=1$, $\ell_p(\nabla^{(\gamma)}_q)=\ell_p(\nabla^{(\gamma)})$ and $\ell_\infty(\nabla^{(\gamma)}_q)=\ell_\infty(\nabla^{(\gamma)})$, as studied by \"Ozger \cite{Ozger}.
	\end{enumerate}
	
	It directly follows from the aforementioned definition of the spaces $\ell_p(\nabla^{(\gamma)}_q)$ and $\ell_\infty(\nabla^{(\gamma)}_q)$ that
	$$\ell_p(\nabla^{(\gamma)}_q)=(\ell_p)_{\nabla^{(\gamma)}_q} \text{ and } \ell_\infty(\nabla^{(\gamma)}_q)=(\ell_\infty)_{\nabla^{(\gamma)}_q}.$$
	
	Accordingly, we propose the following theorem:
	\begin{thm} The following assertions hold true:
		\begin{enumerate}
			\item[(1)] The space $\ell_p(\nabla^{(\gamma)}_q)$ is a $BK$-space attached by norm
			$$\Vert g \Vert_{\ell_p(\nabla^{(\gamma)}_q)}=\Vert h \Vert_{\ell_p}=\left(\sum_{j=0}^\infty \left| \sum_{k=0}^{j}(-1)^{j-k}q^{\binom{j-k}{2}}\dfrac{\Gamma_q(\gamma+1)}{[j-k]_q!\Gamma_q(\gamma-j+k+1)}g_k \right|^p \right)^{1/p}$$
			if $1\leq p<\infty;$ and is a complete $p$-normed space attached by $p$-norm
			$$\Vert g \Vert_{\ell_p(\nabla^{(\gamma)}_q)}=\Vert h \Vert_{\ell_p}=\sum_{j=0}^\infty \left| \sum_{k=0}^{j}(-1)^{j-k}q^{\binom{j-k}{2}}\dfrac{\Gamma_q(\gamma+1)}{[j-k]_q!\Gamma_q(\gamma-j+k+1)}g_k \right|^p $$
			if $0<p<1.$
			\item[(2)] The space $\ell_\infty(\nabla^{(\gamma)}_q)$ is a $BK$-space attached by norm
			$$\Vert g \Vert_{\ell_\infty(\nabla^{(\gamma)}_q)}=\Vert h \Vert_{\ell_\infty}=\sup_{j\in \mathbb{Z}^+_0} \left| \sum_{k=0}^{j}(-1)^{j-k}q^{\binom{j-k}{2}}\dfrac{\Gamma_q(\gamma+1)}{[j-k]_q!\Gamma_q(\gamma-j+k+1)}g_k \right|.$$
		\end{enumerate}
	\end{thm}
	\begin{proof}
		It is a customary practice, and so the detailed proof is excluded.
	\end{proof}
	
	\begin{thm}
		For $\mathfrak X\in \{\ell_p, \ell_\infty \},$ $\mathfrak X(\nabla^{(\gamma)}_q)\cong \mathfrak X.$
	\end{thm}
	\begin{proof}
		Let $\mathfrak X\in \{\ell_p, \ell_\infty \}.$ Define a mapping $\mathcal M: \mathfrak X(\nabla^{(\gamma)}_q)\to \mathfrak X$ by
		$$\mathcal M g=\nabla^{(\gamma)}_q g=h$$
		for all $g\in \mathfrak X(\nabla^{(\gamma)}_q).$ It is not challenging to notice that $\nabla^{(\gamma)}_q$ is a matrix representation of $\mathcal M.$ Moreover $\nabla^{(\gamma)}_q$ is a lower triangular matrix; and therefore $\mathcal M$ is a linear bijection preserving the norm. This concludes the proof.  
	\end{proof}
	
	A normed linear space $\mathfrak X$ endowed with norm $\Vert\cdot\Vert$, is said to possess a Schauder basis $b=(b_j )$ if corresponding to each element $g=(g_j)$ in $\mathfrak X,$ there exists a sequence of scalars $(\alpha_j)$ such that
	$$\lim_{j\to \infty}\left\|g-\sum_{k=0}^j \alpha_k b_k \right\|=0.$$
	We now recall two important Lemmas that can be found in \cite{Basar_Monograph}.
	\begin{lem}\emph{\cite[Theorem 4.1.2]{Basar_Monograph}}\label{BasarLem1}
		Let $(b_j)$ be a Schauder basis of a linear metric sequence space $(\mathfrak X,d)$ and $\Phi$ and $\Psi$ be two triangles such that $\Psi=\Phi^{-1}.$ Then, $\left\{\Psi (b_j)\right\}$ is a Schauder basis of the matrix domain $\mathfrak Y= \mathfrak X_{\Phi}$ with metric $d_{\Phi}$ defined by $d_{\Phi}(u,v)=d(\Phi u,\Phi v)$ for all $u,v\in \mathfrak Y.$ 
	\end{lem}
	\begin{lem}\emph{\cite[Remark 4.1.3]{Basar_Monograph}}\label{BasarLem2}
		Let $\Phi$ be a triangle. Then, the matrix domain $\mathfrak X_{\Phi}$ of a linear metric sequence space $\mathfrak X$ has a basis iff $\mathfrak X$ has a basis.
	\end{lem}
	
	Let $e^{(j)}$ be a sequence with $1$ in the $j^{\textrm{th}}$ position and $0$, elsewhere. It is known that $\left\{e^{(j)}\right\}$ is the Schauder basis of the space $\ell_p$. Since $\nabla^{(\gamma)}_q$ is a triangle, it follows immediately from Lemmas \ref{BasarLem1} and \ref{BasarLem2} that the sequence $\left\{\nabla^{-\gamma}_q(e^{(j)})\right\}_{j\in \mathbb{Z}^+_0}$ forms the Schauder basis for the space $\ell_p(\nabla^{(\gamma)}_q).$ As an application of this fact, the following statements are derived:
	\begin{thm}
		Consider the sequence $\zeta^{(k)}(q)=\big(\zeta^{(k)}_j(q)\big)_{k\in \mathbb{Z}^+_0}$ of the elements of the space $\ell_p(\nabla^{(\gamma)}_q)$ defined in the following manner:
		\begin{eqnarray*}
			\zeta^{(k)}_j(q)=\left\{\begin{array}{ccc}
				(-1)^{j-k}\dfrac{\Gamma_q(-\gamma+1)}{[j-k]_q!\Gamma_q(-\gamma-j+k+1)},& &0\leq k\leq j, \\
				0,&  & \textrm{otherwise}. 
			\end{array}\right.
		\end{eqnarray*}
		Then, each of the statements given below holds true:
		\begin{enumerate}
			\item[(1)] The sequence $$\left\lbrace \zeta^{(k)}(q)\right\rbrace_{k\in \mathbb{Z}^+_0} $$ is basis of $\ell_p(\nabla^{(\gamma)}_q)$ and each $g\in \ell_p(\nabla^{(\gamma)}_q)$ has the unique representation $$g=\sum_{k=0}^{\infty}h_k \zeta^{(k)}(q),$$
			where $h_j=(\nabla^{(\gamma)}_qg)_j$ for each $j\in \mathbb{Z}^+_0.$
			\item[(2)] The space $\ell_\infty(\nabla^{(\gamma)}_q)$ has no Schauder basis.
		\end{enumerate}
	\end{thm}

	\section{Computation of $\mathfrak{X}^\lambda,$ where $\mathfrak X\in \left\{\ell_p(\nabla^{(\gamma)}_q), \ell_\infty(\nabla^{(\gamma)}_q) \right\}$ and $\lambda\in \{\alpha, \beta,\gamma\}$}
	Here, we compute $\alpha$-, $\beta$-, and $\gamma$-duals of $\ell_p(\nabla^{(\gamma)}_q)$ and $\ell_\infty(\nabla^{(\gamma)}_q).$ We inform that the symbol `$\gamma$' used in the $\gamma$-dual has a discrete meaning to that of the one used in the operator $\nabla^{(\gamma)}_q.$  
	Certainly. Here is the rephrased version:
	
	\begin{definition}
		Let $\mathfrak{X}$ and $\mathfrak{Y}$ denote any two sequence spaces. The multiplier space $\mathcal{M}(\mathfrak{X}, \mathfrak{Y})$ is defined as follows:
		\begin{eqnarray*}
			\mathcal{M}(\mathfrak{X}, \mathfrak{Y}) := \{a = (a_j) \in \omega : ag = (a_j g_j) \in \mathfrak{Y} \text{ for every } g = (g_j) \in \mathfrak{X}\}.
		\end{eqnarray*}
		Specifically, $\mathcal{M}(\mathfrak{X}, \ell_1) = \mathfrak{X}^{\alpha}$, $\mathcal{M}(\mathfrak{X}, cs) = \mathfrak{X}^{\beta}$, and $\mathcal{M}(\mathfrak{X}, bs) = \mathfrak{X}^{\gamma}$, where $cs$ and $bs$ represent the spaces of convergent series and bounded series, respectively.
	\end{definition}
	
	Let $\mathcal{Z}$ represent the collection of all finite subsets of $\mathbb{Z}^+_{0}$. The below given assertions, primarily derived from \cite{MT}, are fundamental for calculating the duals.
	\begin{lem}\emph{\cite{GE,LM,MT}}\label{mc} Let $\Phi=(\phi_{jk})$ be an infinite matrix over $\mathbb{C}.$ Then, each of the stated claims is true: 
		\begin{enumerate}
			\item[(1)]  $\Phi\in (\ell_\infty,\ell_\infty)$ iff 
			\begin{eqnarray}
			\label{e37}
			&&\sup_{j\in\mathbb{Z}^+_0}\sum_{k=0}^{\infty}\left| \phi_{jk}\right|<\infty.
			\end{eqnarray}
			\item[(2)] $\Phi\in(\ell_{\infty},c)$ iff
			\begin{eqnarray}\label{es494}
			&&\exists \phi_k\in\mathbb{C}\ni\lim_{j\to\infty}\phi_{jk}=\phi_k~\emph{ for all }~k\in\mathbb{Z}^+_0,\\
			\label{eqs37}
			&&\lim_{j\to\infty}\sum_{k=0}^{\infty}\left|\phi_{jk}\right|=\sum_{k=0}^{\infty}\left|\lim_{j\to\infty}\phi_{jk}\right|.
			\end{eqnarray}
			\item[(3)] $\Phi\in(\ell_{\infty},\ell_1)$ iff
			\begin{eqnarray}\label{e390}
			\sup_{J\in\mathcal{Z}}\sum_{k=0}^{\infty}\left|\sum_{j\in J}\phi_{jk}\right|<\infty.
			\end{eqnarray}
			\item[(4)] $\Phi\in(\ell_p,\ell_\infty)$ iff
			\begin{eqnarray}
			&&\sup_{j\in\mathbb{Z}^+_{0}}\sum_{k=0}^{\infty}| \phi_{jk}|^{p'}<\infty,~ (1<p<\infty).\label{e0391}\\
			\label{eq28}
			&&\sup_{j,k\in\mathbb{Z}^+_{0}}\left|\phi_{jk}\right|^{p}<\infty,~~(0<p\leq 1).
			\end{eqnarray}
			\item[(5)] \begin{enumerate}
				\item[(a)] For $1<p<\infty$, $\Phi\in(\ell_p,c)$ iff each of the conditions \eqref{es494} and  \eqref{e0391}  holds true.
				\item[(b)] For $0<p\leq 1$, $\Phi\in(\ell_p,c)$ iff each of the \eqref{es494} and \eqref{eq28} holds true.
			\end{enumerate}
			\item[(6)]  $\Phi\in (\ell_{p},\ell_{1})$ iff
			\begin{eqnarray}\label{eqal028}
			&&\sup_{J\in\mathcal{Z}}\sup_{k\in\mathbb{Z}^+_{0}}\left|\sum_{j\in J}\phi_{jk}\right|^{p}<\infty,\quad (0 < p\leq 1). \\
			&&\sup_{J\in\mathcal{Z}}\sum_{k=0}^{\infty}\left|\sum_{j\in J}\phi_{jk}\right|^{p'}<\infty,\quad (1<p<\infty).\label{eqal0028}
			\end{eqnarray}
		\end{enumerate}
	\end{lem}
	\begin{thm}
		Define the sets $\mathfrak S(q)$ and $\mathfrak S^{(p')}(q)$ in the following manner: 
		\begin{eqnarray*}
			\mathfrak S(q)&:=&\left\lbrace a=(a_j)\in\omega:  \sup_{J\in\mathcal{Z}}\sup_{k\in\mathbb{Z}^+_{0}}\left|\sum_{j\in J}(-1)^{j-k}\frac{\Gamma_q(-\gamma+1)}{[j-k]_q!\Gamma_q(-\gamma-j+k+1)}a_j\right|^{p}<\infty \right\rbrace;\\
			\mathfrak S^{(p')}(q)&:=&\left\lbrace a=(a_j)\in\omega: \sup_{J\in\mathcal{Z}}\sum_{k=0}^{\infty}\left|\sum_{j\in J}(-1)^{j-k}\frac{\Gamma_q(-\gamma+1)}{[j-k]_q!\Gamma_q(-\gamma-j+k+1)}a_j\right|^{p'}<\infty  \right\rbrace.
		\end{eqnarray*}
		Then, 
		\begin{enumerate}
			\item[(1)]  $\left\lbrace \ell_{p}(\nabla^{(\gamma)}_q) \right\rbrace^\alpha= \left\{\begin{array}{ccc}
			\mathfrak{S}(q)&, & 0< p\leq 1, \\
			\mathfrak{S}^{(p')}(q)& , & 1< p<\infty,
			\end{array}\right.$
			\item[(2)]  $\left\lbrace \ell_{\infty}(\nabla^{(\gamma)}_q)\right\rbrace^\alpha=\mathfrak{S}^{(1)}(q).$
		\end{enumerate}
	\end{thm}
	\begin{proof}
		Let $\mathfrak X\in \left\{\ell_p,\ell_\infty \right\}$ and consider any sequence $a=(a_j)\in \omega.$ Then, we obtain the equality
		\begin{eqnarray}\label{MatrixGq}
		a_jg_j=\sum_{k=0}^{j}(-1)^{j-k}\dfrac{\Gamma_q(-\gamma+1)}{[j-k]_q!\Gamma_q(-\gamma-j+k+1)}h_ka_j=(\Lambda(q)h)_j,
		\end{eqnarray}
		for each $j\in \mathbb{Z}^+_0,$ where $\Lambda(q)=\big(\lambda^{q}_{jk}\big)$ is a triangle defined by
		\begin{eqnarray*}
			\lambda^{q}_{jk}=\left\{\begin{array}{ccc}
				(-1)^{j-k}\dfrac{\Gamma_q(-\gamma+1)}{[j-k]_q!\Gamma_q(-\gamma-j+k+1)}a_j&, & 0\leq k\leq j, \\
				0& , & k>j.
			\end{array}\right.
		\end{eqnarray*}
		We realize from the equality \eqref{MatrixGq} that $ag=(a_jg_j)\in \ell_{1}$ whenever $g=(g_j)\in \mathfrak X(\nabla^{(\gamma)}_q)$ iff $\Lambda(q)h\in \ell_{1}$ whenever $h=(h_j)\in \mathfrak X.$ This indicates that $a=(a_j)\in \left\{ \mathfrak X(\nabla^{(\gamma)}_q)\right\}^\alpha $ iff $\Lambda(q)\in(\mathfrak X, \ell_1).$ Consequently, by applying Clauses 3 and 6 of Lemma \ref{mc}, it is concluded that \\
		For $\mathfrak X=\ell_p:$
		\begin{eqnarray*}
			&&\sup_{J\in\mathcal{Z}}\sup_{k\in\mathbb{Z}^+_{0}}\left|\sum_{j\in J}(-1)^{j-k}\frac{\Gamma_q(-\gamma+1)}{[j-k]_q!\Gamma_q(-\gamma-j+k+1)}a_j\right|^{p}<\infty,\quad (0 < p\leq 1);\\
			&&\sup_{J\in\mathcal{Z}}\sum_{k=0}^{\infty}\left|\sum_{j\in J}(-1)^{j-k}\frac{\Gamma_q(-\gamma+1)}{[j-k]_q!\Gamma_q(-\gamma-j+k+1)}a_j\right|^{p'}<\infty,\quad (1<p<\infty).
		\end{eqnarray*}
		For $\mathfrak X=\ell_\infty:$
		$$\sup_{J\in\mathcal{Z}}\sum_{k=0}^{\infty}\left|\sum_{j\in J}(-1)^{j-k}\frac{\Gamma_q(-\gamma+1)}{[j-k]_q!\Gamma_q(-\gamma-j+k+1)}a_j\right|<\infty.$$
		This concludes the proof.
	\end{proof}
	\begin{thm}\label{thm36}
		Define the sets $\mathfrak T_v(q),$ $v=1,2,3$, in the following manner:
		\begin{eqnarray*}
			\mathfrak T_1(q)&:=&\left\lbrace a=(a_j)\in\omega: \lim_{j\to \infty}(-1)^{j-k}\frac{\Gamma_q(-\gamma+1)}{[j-k]_q!\Gamma_q(-\gamma-j+k+1)}a_j\text{ exists for all $j\in \mathbb{Z}^+_0$} \right\rbrace,\\
			\mathfrak T_2(q)&:=&\left\lbrace  a=(a_j)\in\omega: \sup_{j\in \mathbb{Z}^+_0}\sum_{k=0}^{\infty}\left|(-1)^{j-k}\dfrac{\Gamma_q(-\gamma+1)}{[j-k]_q!\Gamma_q(-\gamma-j+k+1)}a_j\right|^{p'}<\infty\right\rbrace,\\
			\mathfrak T_3(q)&:=&\left\lbrace a=(a_j)\in\omega:  \sup_{j,k\in \mathbb{Z}^+_0}\left|(-1)^{j-k}\dfrac{\Gamma_q(-\gamma+1)}{[j-k]_q!\Gamma_q(-\gamma-j+k+1)}a_j\right|^p<\infty   \right\rbrace,\\
			\mathfrak T_4(q)&:=&\Bigg\{ a=(a_j)\in\omega:  \lim_{j\to \infty}\sum_{k=0}^\infty\left|(-1)^{j-k}\dfrac{\Gamma_q(-\gamma+1)}{[j-k]_q!\Gamma_q(-\gamma-j+k+1)}a_j\right|\\ 
			&\phantom{\mathrel{=}}&= \sum_{k=0}^\infty\left|\lim_{j\to \infty}(-1)^{j-k}\dfrac{\Gamma_q(-\gamma+1)}{[j-k]_q!\Gamma_q(-\gamma-j+k+1)}a_j\right|  \Bigg\}.
		\end{eqnarray*} 
		Then,
		\begin{enumerate}
			\item[(1)]  $\left\lbrace \ell_{p}(\nabla^{(\gamma)}_q) \right\rbrace^\beta= \left\{\begin{array}{ccc}
			\mathfrak{T}_1(q)\cap \mathfrak{T}_3(q)&, & 0< p\leq 1, \\
			\mathfrak{T}_1(q)\cap \mathfrak{T}_2(q)& , & 1< p<\infty.
			\end{array}\right.$
			\item[(2)]  $\left\lbrace \ell_{\infty}(\nabla^{(\gamma)}_q)\right\rbrace^\beta=\mathfrak{T}_1(q)\cap \mathfrak{T}_4(q).$
		\end{enumerate}
	\end{thm}
	
	\begin{proof}
		Let $\mathfrak X\in \{\ell_p,\ell_\infty \}.$ For any sequence $a=(a_j)\in \omega,$ we have the equality
		\begin{eqnarray}
		\sum_{k=0}^{j}a_k g_k&=&\sum_{k=0}^{j}\left\{\sum_{v=0}^{k}(-1)^{k-v}\dfrac{\Gamma_q(-\gamma+1)}{[k-v]_q!\Gamma_q(-\gamma-k+v+1)}h_v\right\}a_k \nonumber\\
		&=&\sum_{k=0}^{j}\left\{\sum_{v=k}^{j}(-1)^{v-k}\dfrac{\Gamma_q(-\gamma+1)}{[v-k]_q!\Gamma_q(-\gamma-v+k+1)}a_v\right\}h_k \nonumber\\
		&=&(\Omega(q)h)_j,  \label{BetaDualeqn}
		\end{eqnarray}
		for each $j\in \mathbb{Z}^+_0,$ where $\Omega(q)=(\omega^{q}_{jk})$ is a triangle defined by
		\begin{eqnarray*}
			\omega^{q}_{jk}=\left\{\begin{array}{ccc}
				\sum_{v=k}^{j}(-1)^{v-k}\dfrac{\Gamma_q(-\gamma+1)}{[v-k]_q!\Gamma_q(-\gamma-v+k+1)}a_v&, & 0\leq k\leq j, \\
				0& , & k>j.
			\end{array}\right.
		\end{eqnarray*}
		It derives from the relation \eqref{BetaDualeqn} that $ag=(a_jg_j)\in cs$ whenever $g=(g_j)\in \mathfrak X(\nabla^{(\gamma)}_q)$ iff $\Omega(q)h\in c$ whenever $h=(h_j)\in \mathfrak X.$ This means that $a=(a_j)\in \left\{ \mathfrak X(\nabla^{(\gamma)}_q)\right\}^\beta $ iff $\Omega(q)\in (\mathfrak X, c).$ Thus, by applying Clauses 2 and 5 of Lemma \ref{mc}, we conclude that:\\
		For $\mathfrak X=\ell_p:$
		\begin{eqnarray}\label{es494*}
		&&\exists l_k\in\mathbb{C}\ni\lim_{j\to\infty}\omega^q_{jk}=l_k~\text{ for all }~k\in\mathbb{Z}^+_0,~(0<p<\infty)\\
		&&\sup_{j\in\mathbb{Z}^+_{0}}\sum_{k=0}^{\infty}| \omega^q_{jk}|^{p'}<\infty,~ (1<p<\infty)\nonumber\\
		&&\sup_{j,k\in\mathbb{Z}^+_{0}}\left|\omega^q_{jk}\right|^{p}<\infty,~~(0<p\leq 1).\nonumber
		\end{eqnarray}
		For $\mathfrak X=\ell_\infty:$
		\begin{eqnarray*}
			&&\eqref{es494*} \text{ holds true, and}\\
			&&\lim_{j\to\infty}\sum_{k=0}^{\infty}\left|\omega^q_{jk}\right|=\sum_{k=0}^{\infty}\left|\lim_{j\to\infty}\omega^q_{jk}\right|.
		\end{eqnarray*}
		Thus, our claim is established.
	\end{proof}
	
	\begin{thm}
		The given claims concerning $\gamma$-dual hold true:
		\begin{enumerate}
			\item[(1)]  $\left\lbrace \ell_{p}(\nabla^{(\gamma)}_q) \right\rbrace^\gamma= \left\{\begin{array}{ccc}
			\mathfrak{T}_3(q),& & 0< p\leq 1, \\
			\mathfrak{T}_2(q),&  & 1< p<\infty,
			\end{array}\right.$
			\item[(2)]  $\left\lbrace \ell_{\infty}(\nabla^{(\gamma)}_q)\right\rbrace^\gamma=\mathfrak{T}_2(q) \text{ with $p'=1$}.$
		\end{enumerate}
	\end{thm}
	\begin{proof}
		By proceeding in the pattern similar to the proof of Theorem \ref{thm36}, we realize from the relation \eqref{BetaDualeqn} that $ag=(a_jg_j)\in bs$ whenever $g=(g_j)\in \mathfrak X(\nabla^{(\gamma)}_q)$ iff $\Omega(q)h\in \ell_\infty$ whenever $h=(h_j)\in \mathfrak X.$ This means that $a=(a_j)\in \left\{ \mathfrak X(\nabla^{(\gamma)}_q)\right\}^\gamma $ iff $\Omega(q)\in (\mathfrak X, \ell_\infty).$ Thus, by using Clauses 1 and 4 of Lemma \ref{mc}, our claims are readily established.
	\end{proof}
	
	\section{Matrix Transformations}
	
	This section is devoted to state and prove a characterization theorem related to matrix transformations to/from the spaces $\ell_p(\nabla^{(\gamma)}_q)$ and $\ell_\infty(\nabla^{(\gamma)}_q)$ from/to any arbitrary space $\mathfrak Y\subset \omega.$ 
	
	\begin{thm}\label{thmmt}
		Let $\mathfrak X\in \{\ell_{p}, \ell_{\infty}\}.$ Then, $\Phi=(\phi_{jk})\in \left(\mathfrak X(\nabla^{(\gamma)}_q),\mathfrak Y\right)$ iff 
		\begin{eqnarray}
		\Psi^{(j)}&=&\big(\psi^{(j)}_{mk}\big)\in (\mathfrak X,c), ~(j\in \mathbb{Z}^+_0),\label{mt1}\\	
		\Psi&=&(\psi_{jk})\in (\mathfrak X,\mathfrak Y),\label{mt2}
		\end{eqnarray}
		where
		\begin{eqnarray*}
			\psi^{(j)}_{mk}&=&\left\{\begin{array}{ccc}
				\sum_{v=k}^{m}(-1)^{v-k}\frac{\Gamma_q(-\gamma+1)}{[v-k]_q!\Gamma_q(-\gamma-v+k+1)}\phi_{jv}, & & 0\leq k\leq m \\
				0,&  & k>m,
			\end{array}\right. \\
			\psi_{jk}&=&\sum_{v=k}^{\infty}(-1)^{v-k}\frac{\Gamma_q(-\gamma+1)}{[v-k]_q!\Gamma_q(-\gamma-v+k+1)}\phi_{jv}
		\end{eqnarray*}
		for each $j,k,m\in \mathbb{Z}^+_0.$
	\end{thm}
	\begin{proof}
		Let $\Phi\in \left(\mathfrak X(\nabla^{(\gamma)}_q),\mathfrak Y\right)$ and choose any $g\in \mathfrak X(\nabla^{(\gamma)}_q).$ Then, the following relation is deduced:
		\begin{eqnarray}\label{Phis=Psit}
		\sum_{k=0}^{m}\phi_{jk}g_{k}&=&\sum_{k=0}^{m}\left\{\sum_{v=0}^{k}(-1)^{k-v}\dfrac{\Gamma_q(-\gamma+1)}{[k-v]_q!\Gamma_q(-\gamma-k+v+1)}h_v\right\} \phi_{jk}\nonumber \\
		&=&\sum_{k=0}^{m}\left\{\sum_{v=k}^{m}(-1)^{v-k}\dfrac{\Gamma_q(-\gamma+1)}{[v-k]_q!\Gamma_q(-\gamma-v+k+1)}\phi_{jv}\right\}h_k\nonumber\\
		&=&\sum_{k=0}^{m}\psi^{(j)}_{mk}h_k
		\end{eqnarray}
		for each $j,m\in \mathbb{Z}^+_0.$ Since $\Phi s$ exists, $\Psi^{(j)}\in (\mathfrak X,c).$ Again, $\Phi g=\Psi h$ as $m\to \infty$ in \eqref{Phis=Psit}. Additionally $\Phi g\in \mathfrak Y,$ therefore $\Psi h\in \mathfrak Y$ concluding $\Psi\in (\mathfrak X,\mathfrak Y).$
		
		Conversely, assume that each of the relations \eqref{mt1} and \eqref{mt2} holds true and choose $g\in \mathfrak X(\nabla^{(\gamma)}_q).$ For each \( j \in \mathbb{Z}^+_0 \), we have \(\Psi_j \in \mathfrak{X}^{\beta}\), which consequently implies that \(\Phi_j \in \left\{ \mathfrak{X}(\nabla^{(\gamma)}_q) \right\}^{\beta}\) for every \( j \in \mathbb{Z}^+_0 \). Utilizing the relationship \(\Phi g = \Psi h\) as \( m \to \infty \) (from \eqref{Phis=Psit}), it can be deduced that \(\Phi \in \left( \mathfrak{X}(\nabla^{(\gamma)}_q), \mathfrak{Y} \right)\).
	\end{proof}
	
	\noindent As a direct consequence of Theorem \ref{thmmt}, we identify specific matrix classes \((\mathfrak{X}, \mathfrak{Y})\) where\\ \(\mathfrak{X} \in \{\ell_{1}(\nabla^{(\gamma)}_q), \ell_{p}(\nabla^{(\gamma)}_q), \ell_{\infty}(\nabla^{(\gamma)}_q)\}\) and \(\mathfrak{Y} \in \{\ell_{1}, c_0, c, \ell_{\infty}\}\). The given statements for each \(j \in \mathbb{Z}^+_0\) are essential to substantiate our results:
	\begin{eqnarray}
	&& \lim_{j\to \infty}\sum_{k=0}^\infty \left|\psi_{jk} \right|=0;\label{mtcAdd}\\
	&&\lim_{m\to \infty}\psi^{(j)}_{mk} \textrm{ exists} ~(k\in \mathbb{Z}^+_0);\label{mtc1}\\
	&&\sup_{m,k}\left|\psi^{(j)}_{mk} \right|<\infty; \label{mtc2}\\
	&&\sup_{m\in \mathbb{Z}^+_0}\sum_{k=0}^{m}\left|\psi^{(j)}_{mk} \right|^{p'}<\infty;\label{mtc3}\\
	&&\lim_{m\to \infty}\sum_{k=0}^m \psi^{(j)}_{mk} ;\label{mtc4}\\
	&&\lim_{m\to \infty}\sum_{k=0}^{m}\left|\psi^{(j)}_{mk} \right|=\sum_{k=0}^{m}\left|\psi_{jk} \right|. \label{mtc5}
	\end{eqnarray}
	
	\begin{lem}\label{LC}
		The necessary and sufficient condition for an infinite matrix $\Phi=(\phi_{jk})\in (\mathfrak X,\mathfrak Y)$ can be interpreted from \emph{Table \ref{tablemtc}}, where $\mathfrak X\in\big\{\ell_{1}(\nabla^{(\gamma)}_q),\ell_{p}(\nabla^{(\gamma)}_q), \ell_{\infty}(\nabla^{(\gamma)}_q)\big\}$ $(1<p<\infty),$ $\mathfrak Y\in \left\lbrace \ell_{1},c_0,c,\ell_{\infty} \right\rbrace $ and
		\begin{table}[h!]
			\centering
			\begin{tabular}{|p{8cm}|p{6cm}|}
				\hline
				\textbf{1.} \eqref{mtc1} and \eqref{mtc2} & \textbf{2.} \eqref{mtc1} and \eqref{mtc3}\\
				\textbf{3.} \eqref{mtc1} and \eqref{mtc5} & \textbf{4.} \eqref{e390} with $\psi_{jk}$ instead of $\phi_{jk}$\\
				\textbf{5.} \eqref{es494} with $\phi_{k}=0$ $\forall k$ and $\psi_{jk}$ in place of $\phi_{jk}$ & \textbf{6.} \eqref{es494} with $\psi_{jk}$ instead of $\phi_{jk}$\\
				\textbf{7.} \eqref{e37} with $\psi_{jk}$ instead of $\phi_{jk}$ & \textbf{8.} \eqref{mtcAdd}\\
				\textbf{9.} \eqref{eqs37} with $\psi_{jk}$ instead of $\phi_{jk}$ & \textbf{10.} \eqref{e0391} with $\psi_{jk}$ instead of $\phi_{jk}$\\
				\textbf{11.} \eqref{eqal028} with $p=1$ and  $\psi_{jk}$ instead of $\phi_{jk}$& \textbf{12.} \eqref{eqal0028} with $\psi_{jk}$ instead of $\phi_{jk}$\\
				\textbf{13.} \eqref{eq28} with $p=1$ and $\psi_{jk}$ instead of $\phi_{jk}$&\\
				\hline
			\end{tabular}
		\end{table}
	\end{lem}
	
	\begin{table}[h!]
		\centering
		\begin{tabular}{|p{2.5cm}|p{2.5cm}|p{3cm}|p{3cm}|p{2.5cm}|}
			\hline
			From$\backslash$ To & $\ell_{1}$ & $c_0$ & $c$ & $\ell_{\infty}$\\
			\hline
			$\ell_{1}(\nabla^{(\gamma)}_q)$ & \textbf{1} \& \textbf{11} &  \textbf{1}, \textbf{5} \& \textbf{13} & \textbf{1}, \textbf{6} \& \textbf{13} & \textbf{1} \& \textbf{13}\\
			$\ell_{p}(\nabla^{(\gamma)}_q)$ & \textbf{2} \& \textbf{12} &  \textbf{2}, \textbf{5} \& \textbf{10} & \textbf{2}, \textbf{6} \& \textbf{10} & \textbf{2} \& \textbf{10}\\
			$\ell_{\infty}(\nabla^{(\gamma)}_q)$ & \textbf{3} \& \textbf{4} &  \textbf{3} \& \textbf{8} & \textbf{3}, \textbf{6} \& \textbf{9} & \textbf{3} \& \textbf{7}\\
			\hline
		\end{tabular}
		\caption{Identification of $(\mathfrak X,\mathfrak Y),$ where $\mathfrak X\in \left\lbrace \ell_{1}(\nabla^{(\gamma)}_q),\ell_{p}(\nabla^{(\gamma)}_q), \ell_{\infty}(\nabla^{(\gamma)}_q)\right\rbrace $ and $\mathfrak Y\in \left\lbrace \ell_{1},c_0,c,\ell_{\infty}\right\rbrace $}
		\label{tablemtc}
	\end{table}
	
	\begin{lem}\emph{\cite[Lemma 5.3]{BA}\label{LemmaAltayBasar}}
		Let $\mathfrak X,\mathfrak Y\subset \omega$, $\Phi$ be an infinite matrix, and $\Psi$ be a triangle. Then, $\Phi\in (\mathfrak X,\mathfrak Y_\Psi)$ iff $\Psi\Phi\in (\mathfrak X,\mathfrak Y).$	
	\end{lem}
	
	\noindent It is apparent that Lemma \ref{LemmaAltayBasar} has several applications. To mention a few, we  identify certain matrix classes from $\mathfrak X\in \big\{ \ell_1, c_0, c,\ell_{\infty}\big\}$ to $\mathfrak Y\in \left\lbrace \ell_{p}(\nabla^{(\gamma)}_q),\ell_{\infty}(\nabla^{(\gamma)}_q)\right\rbrace. $
	
	\noindent The given statements indicated by \textbf{A}$^{\prime}$ and \textbf{B}$^{\prime}$ are fundamental for the next outcome:
	\begin{eqnarray}
	\textbf{A}'&:=&\sup_{j\in \mathbb{Z}^+_0}\sum_{k=0}^\infty\left|\upsilon_{jk} \right|^p<\infty ;\label{2mtc1}\\
	\textbf{B}'&:=&\sup_{K\subset \mathfrak{Z}}\sum_{j=0}^\infty\left|\sum_{k\in K}\upsilon_{jk} \right|^p<\infty \label{2mtc2}
	\end{eqnarray}
	
	\begin{cor}
		Let the matrices $\Phi=(\phi_{jk})$ and $\Upsilon=(\upsilon_{jk})$ be connected by the relation
		\begin{equation}\label{Relation}
		\upsilon_{jk}=\sum_{v=0}^j (-1)^{j-v}\frac{\Gamma_q(\gamma+1)}{[v]_q!\Gamma_q(\gamma-j+v+1)}\phi_{vk}
		\end{equation}
		for each $v,k\in \mathbb{Z}^+_0.$ Then, necessary and sufficient condition for $\Phi \in (\mathfrak X,\mathfrak Y),$ where $\mathfrak X\in\big\{\ell_{1},c_0, c, \ell_{\infty}\big\}$ and $ \mathfrak Y\in \left\lbrace \ell_{p}(\nabla^{(\gamma)}_q),\ell_{\infty}(\nabla^{(\gamma)}_q) \right\rbrace $ $(1<p<\infty)$ can be interpreted from \emph{Table \ref{table2}}, where conditions \emph{\textbf{7}} and \emph{\textbf{13}} may be read from \emph{Lemma \ref{LC}} except that the statement `with $\psi_{jk}$ instead of $\phi_{jk}$' should be read as `with $\upsilon_{jk}$ instead of $\phi_{jk}$
	\end{cor}
	
	\begin{table}[h!]
		\centering
		\begin{tabular}{|p{2.5cm}|p{2.5cm}|p{2.5cm}|}
			\hline
			From$\backslash$ To & $\ell_{p}(\nabla^{(\gamma)}_q)$ & $\ell_{\infty}(\nabla^{(\gamma)}_q)$\\
			\hline
			$\ell_{1}$ & \textbf{A}$^{\prime}$  &\textbf{13}\\
			$c_0$ & \textbf{B}$^{\prime}$ &  \textbf{7}\\
			$c$ & \textbf{B}$^{\prime}$ & \textbf{7}\\
			$\ell_{\infty}$ & \textbf{B}$^{\prime}$  & \textbf{7}\\
			\hline
		\end{tabular}
		\caption{Identification of $(\mathfrak X,\mathfrak Y),$ where $\mathfrak X\in \left\lbrace \ell_{1},c_0, c, \ell_{\infty}\right\rbrace $ and $\mathfrak Y\in \left\lbrace \ell_{p}(\nabla^{(\gamma)}_q),\ell_{\infty}(\nabla^{(\gamma)}_q)\right\rbrace $}
		\label{table2}
	\end{table}
	
	
	\begin{cor}
		Let the matrices $\Sigma=(\sigma_{jk})$ and $\Phi=(\phi_{jk})$ be connected by the relation
		$$\sigma_{jk}=\sum_{v=0}^{j} \phi_{vk}$$
		for each $j,k\in \mathbb{Z}^+_0.$ Then, identification of the matrix class $(\mathfrak X,\mathfrak Y)$ can be interpreted from \emph{Lemma \ref{LC}} with $\sigma_{jk}$ instead of $\phi_{jk}$ in \emph{Table \ref{tablemtc}}, where $\mathfrak X\in \big\{\ell_{1}(\nabla^{(\gamma)}_q), \ell_{p}(\nabla^{(\gamma)}_q),  \ell_{\infty}(\nabla^{(\gamma)}_q) \big\}$, and $\mathfrak Y\in \left\{bs,cs,cs_0\right\}.$ 
	\end{cor}
	
	\begin{cor}
		Let the matrices $\mathfrak C(q)=(c^q_{jk})$ and $\Phi=(\phi_{jk})$ be connected by the relation
		$$c^q_{jk}=\sum_{v=0}^{j} \frac{q^v}{[j+1]_q}\phi_{vk}$$
		for each $j,k\in \mathbb{Z}^+_0.$ Then, identification of the matrix class $(\mathfrak X,\mathfrak Y)$ can be interpreted from \emph{Lemma \ref{LC}} with $c^q_{jk}$ instead of  $\phi_{jk}$ in \emph{Table \ref{tablemtc}}, where $\mathfrak X\in \big\{\ell_{1}(\nabla^{(\gamma)}_q), \ell_{p}(\nabla^{(\gamma)}_q),  \ell_{\infty}(\nabla^{(\gamma)}_q) \big\}$, and $\mathfrak Y\in \left\{\mathcal{X}^q_1, \mathcal{X}^q_{0}, \mathcal{X}^q_{c},\mathcal{X}^q_{\infty}\right\}$ are $q$-Ces\`aro sequence spaces studied by Demiriz and \c Sahin \emph{\cite{DemirizSahin}} and Yaying et al. \emph{\cite{YayingHazarikaMursaleen_q-CesaroSpace}}. 
	\end{cor}
	
	Let us define two new sequence spaces $c(\nabla^{(\gamma)}_q)$ and $c_0(\nabla^{(\gamma)}_q)$ as follows:
	{\small
		\begin{eqnarray*}
			c(\nabla^{(\gamma)}_q)&=&\left\{s=(s_j)\in \omega: t=\nabla^{(\gamma)}_q s=\left(\sum_{k=0}^{j}(-1)^{j-k}q^{\binom{j-k}{2}}\dfrac{\Gamma_q(\gamma+1)}{[j-k]_q!\Gamma_q(\gamma-j+k+1)}s_k \right)_{j\in \mathbb{Z}^+_0}\in c \right\}\\
			c_0(\nabla^{(\gamma)}_q)&=&\left\{s=(s_j)\in \omega: t=\nabla^{(\gamma)}_q s=\left(\sum_{k=0}^{j}(-1)^{j-k}q^{\binom{j-k}{2}}\dfrac{\Gamma_q(\gamma+1)}{[j-k]_q!\Gamma_q(\gamma-j+k+1)}s_k \right)_{j\in \mathbb{Z}^+_0}\in c_0 \right\}.
		\end{eqnarray*}
	}
	It is clear that $c(\nabla^{(\gamma)}_q)=c_{\nabla^{(\gamma)}_q}$ and $c_0(\nabla^{(\gamma)}_q)=(c_0)_{\nabla^{(\gamma)}_q}.$
	\begin{cor}
		Let the matrices $\Upsilon=(\upsilon_{jk})$ and $\Phi=(\phi_{jk})$ be connected by the relation defined as in \eqref{Relation}. Then, identification of the matrix class $(\mathfrak X,\mathfrak Y)$ can be interpreted from \emph{Lemma \ref{LC}} with $\upsilon_{jk}$ instead of  $\phi_{jk}$ in \emph{Table \ref{tablemtc}}, where $\mathfrak X\in \big\{\ell_{1}(\nabla^{(\gamma)}_q), \ell_{p}(\nabla^{(\gamma)}_q),  \ell_{\infty}(\nabla^{(\gamma)}_q) \big\}$, and $\mathfrak Y\in \left\{\ell_1(\nabla^{(\gamma)}_q), c_0(\nabla^{(\gamma)}_q), c(\nabla^{(\gamma)}_q),\ell_\infty(\nabla^{(\gamma)}_q)\right\}$.
	\end{cor}
	\section{Conclusion}
	Recently, some fascinating articles have emerged in the literature concerning $q$-difference sequence spaces of integer order, such as those by Yaying et al. \cite{YHME,YayingHazarikaTripathyMursaleen}, Alotaibi et al. \cite{AlotaibiYayingMohiuddine}, and Ellidokuzoglu and Demiriz \cite{ED}. This paper focuses on the construction of $q$-difference sequence spaces $\ell_p(\nabla^{(\gamma)})$ and $\ell_\infty(\nabla^{(\gamma)})$ of fractional order $\gamma$. Our results extend the existing findings in \cite{AlotaibiYayingMohiuddine,Altay,BA,ED,Ozger}. The primary investigations include examining certain intriguing properties of the difference operator $\nabla^{(\gamma)}_q$, $\alpha$-, $\beta$-, $\gamma$-duals, and matrix transformations.
	
	Define the following sequence spaces:
	{\small
		\begin{eqnarray*}
			c(\nabla^{(\gamma)}_q)&=&\left\{g=(g_j)\in \omega: h=\nabla^{(\gamma)}_q g=\left(\sum_{k=0}^{j}(-1)^{j-k}q^{\binom{j-k}{2}}\dfrac{\Gamma_q(\gamma+1)}{[j-k]_q!\Gamma_q(\gamma-j+k+1)}g_k \right)_{j\in \mathbb{Z}^+_0}\in c \right\}\\
			c_0(\nabla^{(\gamma)}_q)&=&\left\{g=(g_j)\in \omega: h=\nabla^{(\gamma)}_q g=\left(\sum_{k=0}^{j}(-1)^{j-k}q^{\binom{j-k}{2}}\dfrac{\Gamma_q(\gamma+1)}{[j-k]_q!\Gamma_q(\gamma-j+k+1)}g_k \right)_{j\in \mathbb{Z}^+_0}\in c_0 \right\}.
		\end{eqnarray*}
	}
	By adopting a methodology parallel to ours, one can investigate and analyze the properties of the sequence spaces $c(\nabla^{(\gamma)}_q)$ and $c_0(\nabla^{(\gamma)}_q).$
	
	
	

\begin{thebibliography}{99}
		\setlength{\baselineskip}{.45cm}
		
		\bibitem{AlotaibiYayingMohiuddine} A. Alotaibi, T. Yaying, S.A. Mohiuddine, \textit{Sequence spaces and spectrum of $q$-difference operator of second order,} Symmetry, {\bf 14} (6) (2022), 1155.
		
		\bibitem{AktugluBekar} H. Aktu\u glu, \c S. Bekar, \textit{On $q$-Ces\`aro matrix and $q$-statistical convergence,} J. Comput. Appl. Math. {\bf 235} (16) (2011), 4717--4723.
		
		\bibitem{Altay} B. Altay, \textit{On the space of $p$-summable difference sequences of order $m$, $(1 \leq p < \infty),$} Studia Scientiarum Mathematicarum Hungarica, \textbf{43} (4) (2006), 387-402.
		
		\bibitem{AB} B. Altay, F. Ba\c sar, \textit{The matrix domain and the fine spectrum of the difference operator $\Delta$ on the sequence space $\ell_p,$ $(0<p<1)$,} Commun. Math. Anal. \textbf{2} (2) (2007), 1-11.
		
		\bibitem{Askey} R. Askey, \textit{The $q$-Gamma and $q$-Beta functions,} Applicable Analysis, {\bf 8} (2) (1978), 125-141.
		
		\bibitem{Basar_Monograph} F. Ba\c sar, \textit{Summability Theory and Its Applications}, Bentham Science Publishers, e-books. Monographs, \.{I}stanbul, 2012.
		
		\bibitem{BA} F. Ba\c sar, B. Altay, \textit{On the space of sequences of $p$-bounded variation and related matrix mappings}, Ukrainian Mathematical Journal, \textbf{55} (2003), 136-147.
		
		
		
		\bibitem{BD} P. Baliarsingh, S. Dutta, \textit{A unifying approach to the difference operators and their applications,} Bol. Soc. Paran. Mat. \textbf{33} (2015), 49-57.
		
		\bibitem{BD1} P. Baliarsingh, S. Dutta, {\it On the classes of fractional order difference sequence spaces and their matrix transformations,} Appl. Math. Comput., {\bf 250}(2015), 665-674.
		
		
		\bibitem{DemirizSahin} S. Demiriz, A. \c Sahin, \textit{$q$-Ces\`aro sequence spaces derived by $q$-analogues,} Adv. Math. {\bf 5} (2) (2016), 97--110.
		
		\bibitem{DB} S. Dutta, P. Baliarsingh, {\it A note on paranormed difference sequence spaces of fractional order and their matrix transformations,} J. Egypt. Math. Soc., {\bf 22} (2014), 249-253.
		
		\bibitem{ED} H. B. Ellidokuzogl\u u, S. Demiriz, \textit{On some generalized $q$-difference sequence spaces,} AIMS Mathematics, {\bf 8} (8) (2023), 18607-18617.
		
		\bibitem{Et_DifferenceSeqnSpace} M. Et, \textit{On some difference sequence spaces,} Do\v ga-Tr. J. Math. {\bf 17} (1993), 18--24.
		
		\bibitem{EtColak_GeneralizedDiffSeqnSpaces} M. Et, R. \c Colak, \textit{On some generalized difference sequence spaces,} Soochow J. Math., {\bf 21} (4) (1995), 377--386. 
		
		
		\bibitem{GE} K.-G. Grosse-Erdmann, \textit{Matrix transformations between the sequence spaces of Maddox}, J. Math. Anal. Appl. \textbf{180 (1)} (1993), 223-238.
		
		
		\bibitem{KacCheung} V. Kac, P. Cheung, \textit{Quantum Calculus}, Springer, New York, (2002).
		
		
		\bibitem{Kizmaz} H. K\i zmaz, \textit{On certain sequence spaces,} Canad. Math. Bull.,  {\bf 24} (1981), 169-176.
		
		\bibitem{LM} C.G. Lascarides, I.J. Maddox, \textit{Matrix transformations between some classes of sequences}, Proc. Camb. Phil. Soc. \textbf{68} (1970), 99-104. 
		
		\bibitem{MalkowskyParashar} E. Malkowsky, S. Parashar, \textit{Matrix transformations in spaces of bounded and convergent difference sequences of order $m,$} Analysis, {\bf 17} (1997), 87--97.
		
		\bibitem{MursaleenBasar_Monograph} M. Mursaleen, F. Ba\c sar, \textit{Sequence spaces: Topics in modern summability theory}, CRC Press, Boca Raton, 2020.
		
		\bibitem{MursaleenTabassumFatima} M. Mursaleen, S. Tabassum, R. Fatma, \textit{On $q$-Statistical Summability Method and Its Properties}, Iran. J. Sci. Technol. Trans. A Sci. {\bf 46} (2022), 455-460.
		
		\bibitem{Ozger} F. \"{O}zger, {\it Characterisations of compact operators on $\ell_p-$type fractional sets of sequences,} Demonstr. Math., \textbf{52} (2019), 105-115.
		
		\bibitem{Srivastava} H.M. Srivastava, \textit{Operators of basic (or q-) calculus and fractional q-calculus and their applications in geometric function theory of complex analysis}, Iran. J. Sci. Technol. Trans. Sci. {\bf 44} (2020), 327-344.
		
		\bibitem{MT} M. Stieglitz, H. Tietz, \textit{Matrixtransformationen von Folgenr\"{a}umen eine Ergebnis\"{u}bersicht}, Math. Z. {\bf 154} (1977), 1-16.
		
		\bibitem{Yaying} T. Yaying, {\it Paranormed Riesz difference sequence spaces of fractional order,} Kragujevac J. Math. {\bf 46} (2) (2022), 175--191.
		
		\bibitem{YHME} T. Yaying, B. Hazarika, S.A. Mohiuddine, M. Et, \textit{On sequence spaces due to $l$th order $q$-difference operator and its spectrum,} Iran. J. Sci. \textbf{47} (2023), 1271-1281.
		
		\bibitem{YayingHazarikaMursaleen_q-CesaroSpace} T. Yaying, B. Hazarika, M. Mursaleen, \textit{On sequence space derived by the domain of $q$-Ces\`aro matrix in $\ell_p$ space and the associated operator ideal,} J. Math. Anal. Appl. {\bf 493} (1) (2021), 1--17.
		
		\bibitem{YayingHazarikaTripathyMursaleen} T. Yaying, B. Hazarika, B.C. Tripathy, M. Mursaleen, \textit{The spectrum of second order quantum difference operator,} Symmetry, {\bf 14} (3) (2022), 557.
	\end{thebibliography}
\end{document}